\documentclass[12pt]{article}
\usepackage{a4wide}
\usepackage{epsfig}
\usepackage{amsmath,amssymb,amsthm}
\usepackage{latexsym}
\usepackage{color}
\usepackage{enumitem}
\usepackage{authblk}
\usepackage[sort,nocompress]{cite} %citation numbers increase

\def\RR{\mathbb{R}}

\def\calP{{\cal P}}
\def\nknots{N}
\def\hmin{h_\mathrm{min}}

\def\prec{t}
\def\rbar{r}
\def\Cgen{{\mathfrak{C}}}

\def\knots{\Xi}

\newtheorem{theorem}{Theorem}
\newtheorem{lemma}{Lemma}
\newtheorem{proposition}{Proposition}

\theoremstyle{remark}\newtheorem{remark}{Remark}
\theoremstyle{remark}\newtheorem{example}{Example}

\usepackage[para]{footmisc}

\begin{document}

\title{Ritz-type projectors with boundary interpolation properties and explicit spline error estimates}
%\title{A note on the choice of Ritz-type projectors for explicit spline error estimates}
%\title{Error estimates in IGA with explicit constants}
%
\author[1]{Espen Sande\thanks{espen.sande@epfl.ch}}
\author[2]{Carla Manni\thanks{manni@mat.uniroma2.it}}
\author[2]{Hendrik Speleers\thanks{speleers@mat.uniroma2.it}}

\affil[1]{\small Institute of Mathematics, EPFL, Lausanne, Switzerland}
\affil[2]{\small Department of Mathematics, University of Rome Tor Vergata, Italy}

\maketitle

\begin{abstract} %%% Abstract for Numer. Math.
In this paper we construct Ritz-type projectors with boundary interpolation properties in finite dimensional subspaces of the usual Sobolev space and we provide a priori error estimates for them. 
The abstract analysis is exemplified by considering spline spaces and we equip the corresponding error estimates with explicit constants.
This complements our results recently obtained for explicit spline error estimates based on the classical Ritz projectors in
[Numer.\ Math.\ 144(4):889--929, 2020].
\end{abstract}

\section{Introduction}
Error estimates for Ritz projections play an important role in the theoretical analysis of the Ritz-Galerkin approximation of differential problems; see, e.g., \cite{Brenner:2008}. 
In this paper we present an abstract framework for constructing a family of Ritz-type projectors with boundary interpolation properties in finite dimensional
subspaces of the usual Sobolev space 
%on a given interval $(a,b)$
\begin{equation}\label{eq:Hr-def}
H^r(a,b):=\{u\in L^2(a,b) : \partial^\alpha u \in L^2(a,b),\, \alpha=1,\ldots,r\},
\end{equation}
and we provide a priori error estimates for them in $L^2$ and standard Sobolev (semi-)norms. 

Let $\calP_p$ be the space of polynomials of degree less than or equal to $p$.
The framework requires that the considered finite dimensional subspace contains the polynomial space $\calP_{2q-1}$ to ensure interpolation of the function and its derivatives up to order $q-1\leq r-1$ at the two ends of the interval $[a,b]$. 
These Hermite interpolation properties at the boundary enhance and complement the results recently obtained for the classical Ritz projectors in \cite{Sande:2020}. It turns out that the two kinds of projectors are deeply related. For a fixed $q$, the difference between the proposed and the corresponding classical Ritz projection belongs to $\calP_{q-1}$; actually, it equals the $L^2$-projection onto $\calP_{q-1}$ of the error of the new Ritz-type projection (see Proposition~\ref{pro:E}). We also provide bounds in Sobolev semi-norms for this difference.

The abstract construction and the related error estimates are elaborated and discussed for projectors onto spline spaces of arbitrary smoothness defined on arbitrary grids. In this case, explicit constants can be deduced for the error bounds. These constants agree with those obtained in \cite{Sande:2020} for the classical Ritz projectors, and so with the numerical evidence found in the literature that smoother spline spaces exhibit a better approximation
behavior per degree of freedom, even for low smoothness of the functions to be approximated.

Projectors with Hermite interpolation properties at the boundary are of particular interest in practice, because they directly allow for building globally smooth approximants by simply gluing locally constructed ones.
When locally working with polynomials, the most well known of such projectors is probably cubic Hermite interpolation as it results in a local construction of $C^1$ cubic spline interpolants. This can be extended to higher smoothness, say $q-1$, by locally considering polynomials of degree at least $2q-1$; see \cite{Buffa:11} for an application in the context of isogeometric analysis.
An interesting and powerful alternative is to replace polynomials by splines as local approximation spaces. 

A common way to obtain projectors with boundary interpolation properties onto spline spaces (of odd degree) is based on (Hermite) interpolation at the knots, by extending the classical construction of $C^2$ cubic spline interpolants. General error estimates with explicit constants for such projectors have been provided in \cite{Schultz:70}. Our projectors do not have any restriction on the spline degree and numerical evaluations reveal that our estimates improve upon those in \cite{Schultz:70}, often involving much smaller constants (see Section~\ref{sec:Schultz}). Better constants for the same interpolating projectors onto odd degree spline spaces can be found in \cite{Agarwal:94} but they are not explicit in most cases. For maximally smooth spline spaces, the latter constants are explicit but remain larger than those in the present work.
Projectors analogous to those in \cite{Schultz:70,Agarwal:94} have also been investigated in \cite{Ainsworth:2020} for periodic boundary conditions and for Lidstone interpolation at the boundary in the case of maximally smooth spline spaces; in those two cases the same explicit error bounds have been obtained as ours (cf. Remark~\ref{rmk:estimate-simple}).

The key ingredient to get our projectors and the corresponding error estimates is the representation of the considered Sobolev spaces and the approximating spline spaces in terms of integral operators described by suitable kernels, following the approach already exploited in \cite{Sande:2020} and earlier work (see, e.g., \cite{Floater:2018,Sande:2019}).

The remainder of this paper is organized as follows. In Section~\ref{sec:gen} we briefly summarize from \cite{Sande:2020} the
abstract framework we are dealing with.
Section~\ref{sec:interp-projector} describes the general construction of the new family of projectors, analyses their boundary interpolation properties, and provides error estimates for them in $L^2$ and standard Sobolev (semi-)norms. The relation between the proposed and the classical Ritz projector is discussed in Section~\ref{sec:Ritz}, where bounds for standard Sobolev semi-norms of their difference are also provided.
The relevant case of projectors onto spline spaces is elaborated in
Section~\ref{sec:spline}, and particular attention is devoted to error bounds with explicit constants.
Finally, Section~\ref{sec:conclusion} collects some final remarks and highlights possible areas of applications of the presented results.

%%%%%%%%%%%%%
\section{General error estimates}\label{sec:gen}
%%%%%%%%%%%%%
In this section we describe an abstract framework to obtain error estimates for the $L^2$-projection onto spaces defined in terms of integral operators. This section is based on \cite[Section~2]{Sande:2020}.

For real-valued functions $f$ and $g$ we denote the norm and inner product on $L^2(a,b)$ by
\begin{equation*}
\| f\|^2 := (f,f), \quad (f,g) := \int_a^b f(x) g(x) dx,
\end{equation*}
and we consider the Sobolev spaces \eqref{eq:Hr-def}.
% \begin{equation*}
% H^r(a,b):=\{u\in L^2(a,b) : \partial^\alpha u \in L^2(a,b),\, \alpha=1,\ldots,r\}.
% \end{equation*}
Let $K$ be the integral operator defined by integrating from the left,
\begin{equation*}%\label{eq:Kint}
(Kf)(x):=\int_a^xf(y)dy.
\end{equation*}
We denote by $K^*$ the adjoint, or dual, of the operator $K$,
defined by
\begin{equation*}
 (f,K^\ast g) = (Kf, g).
\end{equation*}
One can check that $K^*$ is integration from the right,
\begin{equation*}
(K^*f)(x)=\int_x^bf(y)dy;
\end{equation*}
see, e.g., \cite[Section~7]{Floater:2018}. 

Given any finite dimensional subspace $\mathcal{Z}_0\supseteq\mathcal{P}_0$ of $L^2(a,b)$ and any integral operator $K$, we let $\mathcal{Z}_\prec$ for $\prec\geq 1$ be defined by $\mathcal{Z}_\prec:=\mathcal{P}_0+K(\mathcal{Z}_{\prec-1})$. We further assume that they satisfy the equality
\begin{equation}\label{eq:Xsimpl}
\mathcal{Z}_\prec:=\mathcal{P}_0+K(\mathcal{Z}_{\prec-1}) =\mathcal{P}_0+K^*(\mathcal{Z}_{\prec-1}),
\end{equation}
where the sums do not need to be orthogonal (or even direct). From the definition the following equivalence is easy to check,
\begin{equation*}
\mathcal{P}_{k}\subseteq\mathcal{Z}_{\prec} \quad\Leftrightarrow\quad
\mathcal{P}_{k-q}\subseteq\mathcal{Z}_{\prec-q},
\end{equation*}
for any $q\leq k,t$.
Moreover, let $Z_\prec$ be the $L^2$-projector onto $\mathcal{Z}_\prec$, and define $\Cgen_{\prec,\rbar}\in\RR$ for $\prec,\rbar\geq 0$ to be
\begin{equation*} %\label{eq:Cp}
\Cgen_{\prec,\rbar}:=\|(I-Z_\prec)K^\rbar\|.
\end{equation*}
Observe that $\Cgen_{\prec,0}=1$ and $\Cgen_{0,1}:=\|(I-Z_0)K\|=\|(I-Z_0)K^*\|$. 

The space $H^r(a,b)$ can be described as
\begin{equation}\label{eq:Hr}
H^r(a,b)=\mathcal{P}_{0} + K(H^{r-1}(a,b))=\mathcal{P}_{0} + K^*(H^{r-1}(a,b))
=\mathcal{P}_{r-1}+K^r(H^0(a,b)),
\end{equation}
with $H^0(a,b)=L^2(a,b)$ and $\mathcal{P}_{-1}=\{0\}$.
Thus, any $u\in H^r(a,b)$ is of the form $u=g+K^rf$ for $g\in \mathcal{P}_{r-1}$ and $f\in L^2(a,b)$. This leads to the following error estimate for the $L^2$-projection (see also \cite[Theorem~1]{Sande:2020}).

\begin{lemma}\label{lem:L2}
Let $Z_\prec$ be the $L^2$-projector onto $\mathcal{Z}_\prec$ and assume $\mathcal{P}_{r-1}\subseteq\mathcal{Z}_\prec$. Then, for any $u\in H^r(a,b)$ we have
\begin{equation}\label{ineq:L2}
\|u-Z_\prec u\|\leq  \Cgen_{\prec,r}\|\partial^ru\|.
\end{equation}
\end{lemma}
\begin{proof}
Since $\mathcal{P}_{r-1}\subseteq\mathcal{Z}_\prec$ and using \eqref{eq:Hr}, we have $u=g+K^rf$ for $g\in \mathcal{P}_{r-1}$ and $f\in L^2(a,b)$. Thus,
\begin{equation}\label{ineq:Hr}
\|u-Z_\prec u\|=\|g+K^rf-Z_\prec(g+K^rf)\|= \|(I-Z_\prec) K^rf\| \leq  \Cgen_{\prec,r}\|f\|,
\end{equation}
and the result follows from the identity $\partial^ru=f$.
\end{proof}
By definition of the operator norm, the constant $\Cgen_{\prec,r}$ is the smallest possible constant such that the last inequality in \eqref{ineq:Hr} holds for all $f\in L^2(a,b)$. We thus see from the above proof that whenever $\mathcal{P}_{r-1}\subseteq\mathcal{Z}_\prec$, the constant $\Cgen_{\prec,r}$ is the smallest possible constant such that \eqref{ineq:L2} holds for all $u\in  H^r(a,b)$.

\section{A projection with boundary interpolation}\label{sec:interp-projector}
Similarly to \cite{Sande:2019}, we define a sequence of projection operators $Q_\prec^q:H^q(a,b)\to \mathcal{Z}_{\prec}$, for $q=0,\ldots,\prec$, by $Q_{\prec}^0:=Z_{\prec}$ and
\begin{equation}\label{eq:Qproj}
Q_\prec^qu := u(a) + KQ_{\prec-1}^{q-1}\partial u.
\end{equation}
These projections, by definition, commute with the derivative: $\partial Q_\prec^q=Q_{\prec-1}^{q-1}\partial$.
Note that $\partial^q Q_\prec^q=Z_{\prec-q}\partial^q$ can equivalently be stated as
\begin{equation}\label{eq:Qritz}
(\partial^{q}Q_\prec^q u,\partial^qv) = (\partial^q u, \partial^q v), \quad \forall v\in \mathcal{Z}_\prec,
\end{equation} 
since $\partial^q\mathcal{Z}_\prec=\mathcal{Z}_{\prec-q}$. It is also easy to see that these projections satisfy the interpolation property
\begin{equation}\label{eq:Qinterp}
(Q_\prec^qu)^{(\ell)}(a)= u^{(\ell)}(a),\quad \ell=0,\ldots,q-1.
\end{equation}
In the following lemma we show under which conditions there is interpolation at the other end point of the domain.
\begin{lemma}\label{lem:Qinterp}
If $\mathcal{P}_{2q-\ell-1}\subseteq\mathcal{Z}_\prec$ then the projector $Q_\prec^q$ in \eqref{eq:Qproj} satisfies $(Q_\prec^qu)^{(\ell)}(b)= u^{(\ell)}(b)$ for $\ell=0,\ldots,q-1$.
\end{lemma}
\begin{proof}
We proceed by induction on $\ell$.
We first consider $\ell=q-1$. If we pick $v(x)=x^q/q!$ in \eqref{eq:Qritz} we find that
\begin{align*}
(Q^q_\prec u)^{(q-1)}(b)-(Q^q_\prec u)^{(q-1)}(a) = (\partial^{q} Q^q_\prec u,1) = (\partial^{q}u,1) =  u^{(q-1)}(b)- u^{(q-1)}(a),
\end{align*}
and the result follows from \eqref{eq:Qinterp}. Next we assume that the result is true for $\ell+1,\ell+2,\ldots, q-1$ and consider the case $\ell$. Using integration by parts, $q-\ell-1$ times, we find from \eqref{eq:Qritz} that 
\begin{align}\label{eq:Q-int_by_parts}
(\partial^{\ell+1} Q^q_\prec u,\partial^{2q-\ell-1}v) = (\partial^{\ell+1}u,\partial^{2q-\ell-1}v),
\end{align}
where the boundary terms disappear due to \eqref{eq:Qinterp} and the induction hypothesis. The result now follows by picking $v(x)=x^{2q-\ell-1}/(2q-\ell-1)!$ in \eqref{eq:Q-int_by_parts}.
\end{proof}
%\begin{remark}
In the case $q=1$, the projector $Q^1_\prec$ was already defined in \cite{Takacs:2018} and \cite[Section~8]{Sande:2020}. Its interpolation property was exploited to build globally $C^0$ functions in the context of multi-patch geometries and error estimates were provided for isogeometric multi-patch discretizations. 
In the case $\mathcal{Z}_\prec=\mathcal{P}_{2q-1}$, the projector $Q^q_\prec$ was used in \cite{Buffa:11} to obtain error estimates for spline spaces (with restrictions on the smoothness). We will generalize these results in Section~\ref{sec:spline-Q}; see in particular Example~\ref{ex:buffa}.
%\end{remark}

Using the classical Aubin--Nitsche duality argument we arrive at the following error estimates for $Q^q_\prec$.
\begin{lemma}\label{lem:Q-L2}
Let $Q_\prec^q$ be the projector defined in \eqref{eq:Qproj}. Then,
\begin{equation*} %\label{ineq:RitzL2}
\|u-Q^q_\prec u\| \leq  \Cgen_{\prec-q,q}\|(I-Z_{\prec-q})\partial^q u\|,
\end{equation*}
for all $\prec\geq q$ such that %$\mathcal{P}_{2q-1}\subseteq\mathcal{Z}_{\prec}$.
$\mathcal{P}_{q-1}\subseteq\mathcal{Z}_{\prec-q}$.
\end{lemma}
\begin{proof}
Let $u\in H^q(a,b)$ be given and define $w$ as the solution to the Dirichlet problem
\begin{equation*} %\label{eq:Aubin-Nitsche}
\begin{aligned}
(-1)^{q}\partial^{2q}w&=u-Q_\prec^qu,
\\
w(a)&=w(b)=\cdots=w^{(q-1)}(a)=w^{(q-1)}(b)=0.
\end{aligned}
\end{equation*}
Using integration by parts, $q$ times, together with \eqref{eq:Qinterp} and Lemma~\ref{lem:Qinterp}, we have
\begin{align*}
\|u-Q^q_\prec u\|^2&=(u-Q^q_\prec u,u-Q^q_\prec u)=(u-Q^q_\prec u,(-1)^{q}\partial^{2q}w)
\\
&=(\partial^q(u-Q^q_\prec u),\partial^{q} w)=((I-Z_{\prec-q})\partial^q u),\partial^{q} (w-v)),
\end{align*}
for any $v\in\mathcal{Z}_\prec$, since $((I-Z_{\prec-q})\partial^q u,\partial^qv)=0$.
Next, using $\|u-Q^q_\prec u\|=\|\partial^{2q}w\|$ together with the Cauchy--Schwarz inequality we obtain
\begin{equation}\label{ineq:AN1}
\|u-Q^q_\prec u\|\,\|\partial^{2q}w\|\leq \|(I-Z_{\prec-q})\partial^q u\|\,\|\partial^q (w-v)\|.
\end{equation}
If we let $v=Q_\prec^qw$, then Lemma~\ref{lem:L2} implies that
\begin{equation}\label{ineq:AN2}
\|\partial^q(w-Q^q_\prec w)\|=\|(I-Z_{\prec -q})\partial^qw\|
\leq \Cgen_{\prec-q,q} \|\partial^{2q}w\|,
\end{equation}
since $\mathcal{P}_{q-1}\subseteq\mathcal{Z}_{\prec-q}$.
Combining \eqref{ineq:AN1} and \eqref{ineq:AN2} completes the proof.
\end{proof}

\begin{theorem}\label{thm:Q}
Let $u\in H^r(a,b)$ be given.
For any $q=0,\ldots,r$, let $Q_\prec^q$ be the projector onto $\mathcal{Z}_\prec$ defined in \eqref{eq:Qproj}. Then, for any $\ell=0,\ldots,q$ we have
\begin{equation*}
\|\partial^{\ell}(u-Q^q_\prec u)\|\leq 
\Cgen_{\prec-q,q-\ell} \Cgen_{\prec-q,r-q}\|\partial^ru\|,
\end{equation*}
for all $\prec\geq q$ such that $\mathcal{P}_{r-q-1}\subseteq\mathcal{Z}_{\prec-q}$ and $\mathcal{P}_{q-\ell-1}\subseteq\mathcal{Z}_{\prec-q}$.
\end{theorem}
\begin{proof}
Using the commuting property  $\partial^{\ell}Q^q_\prec=Q^{q-\ell}_{\prec-\ell}\partial^{\ell}$ together with Lemma \ref{lem:Q-L2} we obtain
\begin{align*}
\|\partial^{\ell}(u-Q^q_\prec u)\|&=\|(I-Q^{q-\ell}_{\prec-\ell})\partial^{\ell}u\|\leq \Cgen_{\prec-q,q-\ell}\|(I-Z_{\prec-q})\partial^q u\|,
\end{align*}
since $\mathcal{P}_{q-\ell-1}\subseteq\mathcal{Z}_{\prec-q}$. The result now follows from Lemma \ref{lem:L2} since $\mathcal{P}_{r-q-1}\subseteq\mathcal{Z}_{\prec-q}$.
\end{proof}

In \cite[Section~3.1]{Sande:2019} a closely related sequence of projection operators were studied. Let $\widetilde{Q}_\prec^q:H^q(a,b)\to \mathcal{Z}_{\prec}$, for $q=0,\ldots,\prec$, be defined by $\widetilde{Q}_\prec^0:=Z_{\prec}$ and
\begin{equation*}%\label{eq:otherQproj}
\widetilde{Q}_\prec^qu := c(u) + KQ_{\prec-1}^{q-1}\partial u,
\end{equation*}
where $c(u)\in\RR$ is chosen such that $(\widetilde{Q}_\prec^q u,1)=(u,1)$. In the case of $\mathcal{Z}_{\prec}$ being a spline space, the projector $\widetilde{Q}_\prec^1$ was also studied in \cite{Takacs:2016}. In both works \cite{Sande:2019,Takacs:2016}, only maximally smooth splines were considered; here we consider a more general context that allows for splines of any smoothness.
\begin{proposition}\label{pro:otherQproj}
If $\mathcal{P}_{2q-\ell}\subseteq\mathcal{Z}_\prec$ then the projector $Q_\prec^q$ in \eqref{eq:Qproj} satisfies $(\partial^{\ell}Q^q_\prec u,1) = (\partial^{\ell}u,1)$ for $\ell=0,\ldots,q$. Consequently, if $\mathcal{P}_{2q}\subseteq\mathcal{Z}_\prec$ then $Q_\prec^q=\widetilde{Q}_\prec^q$. 
\end{proposition}
\begin{proof}
We first consider the case $\ell>0$. 
Since $\mathcal{P}_{2q-\ell}\subseteq\mathcal{Z}_\prec$ it follows from \eqref{eq:Qinterp} and Lemma~\ref{lem:Qinterp} that
\begin{equation*}
(\partial^{\ell}Q^q_\prec u,1) =(Q^q_\prec u)^{(\ell-1)}(b)-(Q^q_\prec u)^{(\ell-1)}(a) =  u^{(\ell-1)}(b)- u^{(\ell-1)}(a) = (\partial^{\ell}u,1).
\end{equation*}
If $\ell=0$ then, using integration by parts together with \eqref{eq:Qinterp} and Lemma~\ref{lem:Qinterp}, we have
\begin{equation*}
(u-Q^q_\prec u,1) = \frac{(-1)^q}{q!}(\partial^q(u-Q^q_\prec u),x^q) = 0,
\end{equation*}
since $x^{2q}\in\mathcal{Z}_\prec$ and $Q^q_\prec$ satisfies \eqref{eq:Qritz}.
\end{proof}

The degree of the polynomial inclusion in Proposition~\ref{pro:otherQproj} is sharp as illustrated in the following example.
\begin{example}\label{ex:otherQproj}
Let $q=1$ and $[a,b]=[0,1]$. Then, for $\mathcal{Z}_1=\mathcal{P}_1$ ($t=1$) we have
\begin{equation*}
Q^1_1u(x) = u(0) + \int_0^x Z_0\partial u(y)dy = u(0)+x(u(1)-u(0)).
\end{equation*}
It is clear that in general $(Q^1_1u,1)\neq(u,1)$ and thus $Q^1_1$ is different from $\widetilde{Q}^1_1$. On the other hand, for $\mathcal{Z}_2=\mathcal{P}_2$ ($t=2$) we have
\begin{align*}
Q^1_2u(x) &= u(0) + \int_0^x Z_1\partial u(y)dy\\
&=u(0) -2x \left[u(1)+2u(0)-3(u,1)\right] +3x^2 \left[u(1)+u(0)-2(u,1)\right],
\end{align*}
which satisfies $(Q^1_2u,1)=(u,1)$ and thus $Q^1_2$ equals $\widetilde{Q}^1_2$.
\end{example}

We observe that the polynomial condition $\mathcal{P}_{2q-\ell}\subseteq\mathcal{Z}_\prec$ in Proposition~\ref{pro:otherQproj} is slightly more restrictive than the polynomial condition $\mathcal{P}_{2q-\ell-1}\subseteq\mathcal{Z}_{\prec}$ (which is equivalent to $\mathcal{P}_{q-\ell-1}\subseteq\mathcal{Z}_{\prec-q}$) required in Theorem~\ref{thm:Q}. Hence, when $\mathcal{P}_{2q}\subseteq\mathcal{Z}_\prec$, the explicit error estimates in the theorem are also valid for the projector $\widetilde{Q}_\prec^q$ for all $\ell=0,\ldots,q$. This extends the results in \cite{Sande:2019} that only provides error estimates for $\ell=q-1,q$, and only covers maximally smooth spline spaces. 

\section{Relation to the Ritz projection}\label{sec:Ritz}
In this section we focus on the Ritz projector $R^q_\prec$ considered in \cite{Sande:2020}, and investigate the relation with our projector $Q^q_\prec$.
We recall from \cite[Eq.~(13)]{Sande:2020} that $R_\prec^q: H^q(a,b)\to \mathcal{Z}_\prec$, for $q=0,\ldots,\prec$, is defined by
\begin{equation}\label{eq:Rproj}
\begin{aligned}
(\partial^{q}R_\prec^q u,\partial^qv) &= (\partial^q u, \partial^q v), \quad &&\forall v\in \mathcal{Z}_\prec,
\\
(R_\prec^qu,g)&=(u,g), &&\forall g\in \mathcal{P}_{q-1}.
\end{aligned}
\end{equation}
We first show that the difference between $Q^q_\prec u$ and $R^q_\prec u$ is a polynomial of at most degree $q-1$. 
% \begin{proposition}\label{pro:E}
% Let $Q^q_\prec u$ and $R^q_\prec u$ be defined in \eqref{eq:Qproj} and \eqref{eq:Rproj}, respectively. Then,
% \begin{equation*}
%   E_{q-1}u  :=  R^q_\prec u - Q^q_\prec u \in \mathcal{P}_{q-1},
% \end{equation*}
% such that 
% \begin{equation}\label{eq:Eprop}
%   (E_{q-1}u,g) = (u-Q^q_\prec u,g), \quad \forall g\in \mathcal{P}_{q-1}.
% \end{equation}
% \end{proposition}
\begin{proposition}\label{pro:E}
Let $Q^q_\prec$ and $R^q_\prec$ be the projectors defined in \eqref{eq:Qproj} and \eqref{eq:Rproj}, respectively. Let $P_{q-1}$ be the $L^2$-projector onto $\mathcal{P}_{q-1}$. Then,
\begin{equation*}
  E_{q-1}u  :=  R^q_\prec u - Q^q_\prec u = P_{q-1}(u - Q^q_\prec u)\in \mathcal{P}_{q-1}.
\end{equation*}
\end{proposition}
\begin{proof}
%The property in \eqref{eq:Eprop} follows directly from the definition of $E_{q-1}$ and $R^q_\prec$. It remains to prove that $E_{q-1}u\in\mathcal{P}_{q-1}$. Let $g_e$ be the $L^2$-projection of $u-Q^q_\prec u$ onto $\mathcal{P}_{q-1}$, and set $r_e:=g_e+Q^q_\prec u$.
Let $g_e:=P_{q-1}(u - Q^q_\prec u)$ and $r_e:=g_e+Q^q_\prec u$.
Then, from \eqref{eq:Qritz} we obtain
\begin{equation*}
(\partial^{q}r_e,\partial^qv) = (\partial^{q}g_e+\partial^{q}Q^q_\prec u,\partial^qv)=(\partial^q u, \partial^q v), \quad \forall v\in \mathcal{Z}_\prec,
\end{equation*}
and by the definition of $g_e$ we have
\begin{equation*}
  (r_e,g)=(g_e+Q^q_\prec u-u+u,g)=(u,g), \quad \forall g\in \mathcal{P}_{q-1}.
\end{equation*}
Hence, we conclude that $R^q_\prec u=r_e$ and $E_{q-1}u=g_e$.
\end{proof}
Observe that Proposition~\ref{pro:E} implies that
\begin{equation*}
\|u - R^q_\prec u\|^2=\|u - Q^q_\prec u\|^2-\|E_{q-1} u\|^2,
\end{equation*}
and hence
\begin{equation}\label{eq:R-Q-ineq}
\|u - R^q_\prec u\|\leq\|u - Q^q_\prec u\|.
\end{equation}
It is also clear that 
\begin{equation}\label{eq:R-Q-diffq-ineq}
\|\partial^{q}(u - R^q_\prec u)\|=\|\partial^{q}(u - Q^q_\prec u)\|.
\end{equation}
\begin{example}
Let $q=1$ and $r\geq1$. Then, for $u\in H^r(a,b)$ and $\ell=0,1$ we have
\begin{equation*}
\|\partial^{\ell}(u-R^1_\prec u)\|\leq \|\partial^{\ell}(u-Q^1_\prec u)\|\leq 
\Cgen_{\prec-1,1-\ell} \Cgen_{\prec-1,r-1}\|\partial^ru\|,
\end{equation*}
for all $\prec\geq 1$ such that $\mathcal{P}_{r-2}\subseteq\mathcal{Z}_{\prec-1}$.
This follows immediately from \eqref{eq:R-Q-ineq}--\eqref{eq:R-Q-diffq-ineq} and Theorem~\ref{thm:Q}.
\end{example}

In general, the projectors $Q^q_\prec$ and $R^q_\prec$ are different, but they coincide when polynomials of high enough degree are in the space $\mathcal{Z}_\prec$.

\begin{proposition}\label{pro:R-Q}
Let $Q^q_\prec$ and $R^q_\prec$ be the projectors defined in \eqref{eq:Qproj} and \eqref{eq:Rproj}, respectively. 
If $\mathcal{P}_{2q+i}\subseteq\mathcal{Z}_\prec$ then $(u-Q^q_\prec u,x^{i})=0$ for $i=0,\ldots,q-1$. Consequently, if $\mathcal{P}_{3q-1}\subseteq\mathcal{Z}_\prec$ then $Q^q_\prec=R^q_\prec$.
\end{proposition}
\begin{proof}
Using integration by parts together with \eqref{eq:Qinterp} and Lemma~\ref{lem:Qinterp}, we have
\begin{align*}
(u-Q^q_\prec u,x^i) =\frac{(-1)^q i!}{(q+i)!}(\partial^q(u-Q^q_\prec u), x^{q+i}) = 0,
\end{align*}
since $Q^q_\prec$ satisfies \eqref{eq:Qritz} and $x^{2q+i}\in\mathcal{Z}_\prec$ for all $i\in\{0,\ldots,q-1\}$.
\end{proof}

The degree of the polynomial inclusion in Proposition~\ref{pro:R-Q} is sharp as illustrated in the following example.
\begin{example}\label{ex:R-Q}
Let $q=2$ and $[a,b]=[0,1]$. We choose $u(x)=x^6$.
Then, for $\mathcal{Z}_2=\mathcal{P}_2$ ($t=2$) we have
\begin{equation*}
  Q^2_2 u(x) = 3x^2, \quad
  R^2_2 u(x) = Q^2_2 u(x) - \frac{33}{14}x + \frac{9}{28};
\end{equation*}
and for $\mathcal{Z}_3=\mathcal{P}_3$ ($t=3$) we have
\begin{equation*}
  Q^2_3 u(x) = 4x^3 - 3x^2, \quad
  R^2_3 u(x) = Q^2_3 u(x) + \frac{3}{70}x + \frac{17}{140};
\end{equation*}
and for $\mathcal{Z}_4=\mathcal{P}_4$ ($t=4$) we have
\begin{equation*}
  Q^2_4 u(x) = \frac{30}{7}x^4 - \frac{32}{7}x^3 + \frac{9}{7}x^2, \quad
  R^2_4 u(x) = Q^2_4 u(x) + \frac{3}{70}x - \frac{3}{140};
\end{equation*}
and for $\mathcal{Z}_5=\mathcal{P}_5$ ($t=5$) we have
\begin{equation*}
  Q^2_5 u(x) = R^2_5 u(x) = 3x^5 - \frac{45}{14}x^4 + \frac{10}{7}x^3 - \frac{3}{14}x^2.
\end{equation*}
Note that $Q^2_tu$ is interpolating $u$ at $b=1$ only for $t\geq3$, in accordance with Lemma~\ref{lem:Qinterp}.
\end{example}

One can ask the question of whether the Ritz projector $R^q_\prec$ satisfies other boundary conditions given a polynomial inclusion $\mathcal{P}_{p}\subseteq\mathcal{Z}_\prec$ of degree $p<3q-1$. We consider this problem in the next example.

\begin{example}
Let $q=2$ and define $f(x) := u(x) - R^2_\prec u(x)$. Using integration by parts twice we have
\begin{equation*}
0=(\partial^2f,\partial^2v) %= -(\partial f,\partial^3v) + \left[\partial f \partial^2 v\right]_a^b
=(f,\partial^4v) + \left[\partial f \partial^2 v\right]_a^b - \left[f \partial^3 v\right]_a^b, \quad v\in\mathcal{Z}_\prec.
\end{equation*}
We now pick $v(x)=x^s/s!$ for $s=2,3,4$ to obtain
\begin{align*}
\left[f'(x) \right]_a^b&= 0, 
\\
\left[ f'(x) x\right]_a^b - \left[f(x) \right]_a^b&= 0,
\\
\left[f'(x) x^2/2\right]_a^b - \left[f(x) x\right]_a^b&= 0,
\end{align*}
where we have used that $(f,1)=0$ in the last equation. This implies that
\begin{alignat*}{3}
f'(b)&= f'(a), \quad &\mathcal{P}_{2}\subseteq\mathcal{Z}_\prec,
\\
f'(b) b-f'(a)a &= f(b)-f(a), \quad &\mathcal{P}_{3}\subseteq\mathcal{Z}_\prec,
\\
f'(b)b^2-f'(a)a^2  &=2\left(f(b)b-f(a)a\right), \quad &\mathcal{P}_{4}\subseteq\mathcal{Z}_\prec.
\end{alignat*}
By combining the above conditions we also deduce
\begin{alignat*}{3}
f'(b) &= \frac{f(b)-f(a)}{b-a}, \quad &\mathcal{P}_{3}\subseteq\mathcal{Z}_\prec,
\\
f(b) &= -f(a), \quad &\mathcal{P}_{4}\subseteq\mathcal{Z}_\prec.
\end{alignat*}
Thus, if $\mathcal{P}_{2}\subseteq\mathcal{Z}_\prec$, then we have
$$(R^2_\prec u)'(b)-(R^2_\prec u)'(a) = u'(b)-u'(a).$$
If $\mathcal{P}_{3}\subseteq\mathcal{Z}_\prec$, then we also have
$$(R^2_\prec u)'(b) - \frac{R^2_\prec u(b)-R^2_\prec u(a)}{b-a} = u'(b) - \frac{u(b)-u(a)}{b-a}.$$
Finally, if $\mathcal{P}_{4}\subseteq\mathcal{Z}_\prec$, then we also have
$$R^2_\prec u(b) + R^2_\prec u(a) = u(b) + u(a).$$
These conditions can be verified with the function and the Ritz projectors considered in Example~\ref{ex:R-Q}.
We may conclude that the Ritz projector $R^2_\prec$
satisfies some type of boundary relations for various polynomial inclusions, but they do not seem very useful in practice unless the polynomial degree is high enough to be covered by Proposition \ref{pro:R-Q}.
\end{example}

To complement the result in Proposition~\ref{pro:R-Q}, we also provide estimates for the difference between $Q^q_\prec$ and $R^q_\prec$ in general. 
We start with the following inverse inequality (see \cite{Goetgheluck:90} for details and other similar results).
\begin{lemma}\label{lem:inv-ineq-poly}
If $g\in \mathcal{P}_{p}$ then we have
\begin{align*}
\|\partial g\| \leq \frac{d_p}{b-a}\|g\|, \quad d_p:=\sqrt{\frac{p(p+1)(p+2)(p+3)}{2}}.
\end{align*}
\end{lemma}

\begin{proposition}\label{pro:R-Q-bound}
Let $Q^q_\prec$ and $R^q_\prec$ be the projectors defined in \eqref{eq:Qproj} and \eqref{eq:Rproj}, respectively. We have
\begin{equation*}
  \|R^q_\prec u - Q^q_\prec u\| \leq \|u - Q^q_\prec u\|.
\end{equation*}
Furthermore, for $\ell=1,\ldots,q-1$, we have
% \begin{equation*}
%   \|\partial^\ell(R^q_\prec u - Q^q_\prec u)\| \leq \left(\frac{2\sqrt{3}}{b-a}\right)^{\ell}\left(\prod_{i=q-\ell}^{q-1} i^2\right)\|R^q_\prec u - Q^q_\prec u\|.
% \end{equation*}
\begin{equation*}
  \|\partial^\ell(R^q_\prec u - Q^q_\prec u)\| \leq \left(\frac{1}{b-a}\right)^{\ell}\left(\prod_{i=q-\ell}^{q-1} d_i\right)\|R^q_\prec u - Q^q_\prec u\|.
\end{equation*}
Finally, for any $\ell\geq q$, we have
\begin{equation*}
  \|\partial^\ell(R^q_\prec u - Q^q_\prec u)\|=0.
\end{equation*}
\end{proposition}
\begin{proof}
This follows immediately from Proposition~\ref{pro:E}, $\|P_{q-1}(u - Q^q_\prec u)\|\leq\|u - Q^q_\prec u\|$, and a repeated application of Lemma~\ref{lem:inv-ineq-poly}.
\end{proof}

Let $u\in H^r(a,b)$ be given. The results in %\eqref{eq:R-Q-ineq}--\eqref{eq:R-Q-diffq-ineq} and 
Propositions~\ref{pro:R-Q} and~\ref{pro:R-Q-bound} can be combined with Theorem~\ref{thm:Q} to achieve an estimate for (derivatives of) the difference between $R_\prec^q u$ and $Q^q_\prec u$. In particular, we have
\begin{equation}\label{eq:R-error}
\|R^q_\prec u - Q^q_\prec u\|
\leq \Cgen_{\prec-q,q}\Cgen_{\prec-q,r-q} \|\partial^ru\|,
\end{equation}
and
% \begin{equation}\label{eq:R-error-l}
% \|\partial^\ell(R^q_\prec u - Q^q_\prec u)\|
% \leq \Cgen_{\prec-q,q}\Cgen_{\prec-q,r-q}\left(\frac{2\sqrt{3}}{b-a}\right)^{\ell}\left(\prod_{i=q-\ell}^{q-1} i^2\right) \|\partial^ru\|, \quad 1\leq\ell< q,
% \end{equation}
\begin{equation}\label{eq:R-error-l}
\|\partial^\ell(R^q_\prec u - Q^q_\prec u)\|
\leq \Cgen_{\prec-q,q}\Cgen_{\prec-q,r-q}\left(\frac{1}{b-a}\right)^{\ell}\left(\prod_{i=q-\ell}^{q-1} d_i\right) \|\partial^ru\|, \quad 1\leq\ell< q,
\end{equation}
for all $\prec\geq q$ such that $\mathcal{P}_{r-q-1}\subseteq\mathcal{Z}_{\prec-q}$ and $\mathcal{P}_{q-1}\subseteq\mathcal{Z}_{\prec-q}$. 
%In Section~\ref{sec:spline} we will show that the estimate in \eqref{eq:R-error-l} has a superconvergent behavior in the context of splines.

\section{Spline spaces}\label{sec:spline}
For $k\geq0$, let $C^k[a,b]$ be the classical space of functions with continuous derivatives of order $0,1,\ldots,k$ on the interval $[a,b]$.
We further let $C^{-1}[a,b]$ denote the space of bounded, piecewise continuous functions on $[a,b]$ that are discontinuous only at a finite number of points.

Suppose $\knots:= (\xi_0,\ldots,\xi_{\nknots+1})$ is a sequence of (break) points
such that
\begin{equation*}%\label{knots}
a=:\xi_0 < \xi_1 < \cdots < \xi_{\nknots} < \xi_{\nknots+1}:= b,
\end{equation*}
and let
\begin{equation*}%\label{eq:hmax}
h:=\max_{j=0,\ldots,\nknots} (\xi_{j+1}-\xi_j).
\end{equation*}
Moreover, set $I_j := [\xi_j,\xi_{j+1})$,
$j=0,1,\ldots,\nknots-1$, and $I_\nknots := [\xi_\nknots,\xi_{\nknots+1}]$.
%For any $p \geq 0$, let $\calP_p$ be the space of polynomials of degree at most $p$. 
Then, for $-1\leq k\leq p-1$, we define the space $\mathcal{S}^k_{p,\knots}$ of splines of degree $p$ and smoothness $k$ by
\begin{equation*}
\mathcal{S}^k_{p,\knots} := \{s \in C^{k}[a,b] : s|_{I_j} \in \calP_p,\, j=0,1,\ldots,\nknots \}.
\end{equation*}
%and we set
%\begin{equation*}
%\mathcal{S}_{p,\knots} := \mathcal{S}^{p-1}_{p,\knots}.
%\end{equation*}
With a slight misuse of terminology, we will refer to $\knots$ as knot sequence and to its elements as knots.
We use the notation $S_p^k: L^2(a,b)\to \mathcal{S}_{p,\knots}^k$ 
%and  $S_p: L^2(a,b)\to \mathcal{S}_{p,\knots}$ 
for the $L^2$-projector onto spline spaces.
%while $P_p: L^2(a,b)\to \calP_p$ stands for the $L^2$-projector onto the polynomial space $\mathcal{P}_{p}$.

Define now the constant $c_{p,k,r}$ for $p\geq r-1$ as follows. 
If $k=p-1$, we let
\begin{equation*}
c_{p,p-1,r}:=\left(\frac{1}{\pi}\right)^r,
\end{equation*} 
and if $k\leq p-2$, we~let 
\begin{equation*}
c_{p,k,r}:=\left(\frac{1}{2}\right)^{r}\begin{cases}
\left(\dfrac{1}{\sqrt{(p-k)(p-k+1)}}\right)^r, &k\geq r-2,\\[0.5cm]
\left(\dfrac{1}{\sqrt{(p-k)(p-k+1)}}\right)^{k+1}\sqrt{\dfrac{(p+1-r)!}{(p-1+r-2k)!}}, & k<r-2.
\end{cases}
\end{equation*}
The next result was then shown in \cite[Theorem~3]{Sande:2020}.
\begin{theorem}\label{thm:low-order}
Let $u\in H^r(a,b)$ be given. For any knot sequence $\knots$, let $S^k_p$ be the $L^2$-projector onto $\mathcal{S}^k_{p,\knots}$ for $-1\leq k\leq p-1$. Then,
\begin{equation*}
\|u-S^k_pu\|\leq c_{p,k,r}h^r\|\partial^ru\|,
\end{equation*}
for all $p\geq r-1$.
\end{theorem}

\begin{remark}\label{rmk:c-bound}
In \cite[Remark~3]{Sande:2020} it was observed that for $k\leq p-2$ the constant $c_{p,k,r}$ can be bounded by a simpler expression using the Stirling formula. For $k\geq r-2$, we have
\begin{equation*}
c_{p,k,r}\leq \left(\frac{1}{2(p-k)}\right)^r,
\end{equation*}
while for $k<r-2$ we get
\begin{equation*}%\label{ineq:low-smoothness}
c_{p,k,r}\leq \left(\frac{e}{4(p-k)}\right)^r.
\end{equation*}
\end{remark}

In the following subsections we detail the construction and properties of the projection in \eqref{eq:Qproj} when the approximation space is the spline space $\mathcal{S}^k_{p,\knots}$.

\subsection{The Q projector for spline spaces}\label{sec:spline-Q}

If $\mathcal{Z}_0=\mathcal{S}^{-1}_{p-k-1,\knots}$ then we have
$\mathcal{Z}_{k+1}=\mathcal{S}^k_{p,\knots}$
for the sequence of spaces in \eqref{eq:Xsimpl}.
Specifically,
\begin{equation*}
\mathcal{S}^k_{p,\knots} = \mathcal{P}_0+K(\mathcal{S}^{k-1}_{p-1,\knots}) = \mathcal{P}_0+K^*(\mathcal{S}^{k-1}_{p-1,\knots}), \quad k\geq0.
\end{equation*}
As a special case of the projection in \eqref{eq:Qproj} we define the sequence of projection operators $Q_p^{q,k}:H^q(a,b)\to \mathcal{S}^k_{p,\knots}$, for $q=0,\ldots,k+1$, by $Q_{p}^{0,k}:=S^k_p$ and
\begin{equation}\label{eq:Qspline}
Q_p^{q,k}u := u(a) + KQ_{p-1}^{q-1,k-1}\partial u.
\end{equation}
By combining Theorem~\ref{thm:low-order} with Theorem~\ref{thm:Q} we obtain the following error estimate for $Q_p^{q,k}$.

\begin{theorem}\label{thm:Qspline}
Let $u\in H^r(a,b)$ be given.
For any degree $p$, knot sequence $\knots$ and smoothness $-1\leq k\leq p-1$, let $Q_p^{q,k}$ be the projector onto $\mathcal{S}^{k}_{p,\knots}$ defined in \eqref{eq:Qspline} for $q=0,\ldots,\min\{k+1,r\}$. Then, for any $\ell=0,\ldots,q$, we have
\begin{equation*}
\|\partial^{\ell}(u-Q^{q,k}_pu)\| \leq c_{p-q,k-q,q-\ell} c_{p-q,k-q,r-q} h^{r-\ell}\|\partial^ru\|,
\end{equation*}
for all $p\geq \max\{r-1,2q-\ell-1\}$. %$p\geq \max\{q,r-1,2q-\ell-1\}$.
\end{theorem}

\begin{remark}\label{rmk:estimate-simple}
In the case of maximal smoothness, $k=p-1$, the estimate in the above theorem becomes
\begin{equation}\label{eq:spline-max-smooth}
\|\partial^{\ell}(u-Q^{q,k}_pu)\| \leq \left(\frac{h}{\pi}\right)^{r-\ell}\|\partial^ru\|.
\end{equation}
If, on the other hand, $k\leq p-2$, we can use the simplified estimates in Remark~\ref{rmk:c-bound} to obtain
\begin{equation*}
\|\partial^{\ell}(u-Q^{q,k}_pu)\| \leq \left(\frac{e\,h}{4(p-k)}\right)^{r-\ell}\|\partial^ru\|.
\end{equation*}
\end{remark}

\begin{remark}\label{rmk:estimate-abstract}
The error bound in Theorem~\ref{thm:Qspline} relies on the constant $c_{p,k,r}$ used in Theorem~\ref{thm:low-order} for the $L^2$-projector. It is clear from the general result in Theorem~\ref{thm:Q} and \cite[Remark~1]{Sande:2020} that the same procedure can be followed to convert any constant derived for the $L^2$-projector into a constant for the projector $Q^{q,k}_p$. For example, the constant $C_{h,p,k,r}$ for the $L^2$-projector in \cite[Corollary~1]{Sande:2020} immediately leads to an alternative error bound for $Q^{q,k}_p$: under the assumptions of Theorem~\ref{thm:Qspline}, we have
\begin{equation*}
\|\partial^{\ell}(u-Q^{q,k}_pu)\| \leq C_{h,p-q,k-q,q-\ell} C_{h,p-q,k-q,r-q}\|\partial^ru\|.
\end{equation*}
\end{remark}

\begin{example}\label{ex:error}
Let $q=2$ and $[a,b]=[0,1]$. We choose $u(x)=\sin(4x)$.
Figure~\ref{fig:error} shows the convergence behavior of $\|\partial^{\ell}(u-Q^{2,p-1}_pu)\|$ for $p=2,3,4$ and $\ell=0,1$. We observe the optimal convergence order in $h$ (namely $p+1-\ell$) when $p\geq3-\ell$, which agrees with the theoretical prediction in Theorem~\ref{thm:Qspline} ($r=p+1$). Contrarily, the case $p=2$ and $\ell=0$ presents a suboptimal convergence order in $h$ of $2$ (instead of $3$), but this case is not covered by Theorem~\ref{thm:Qspline}.
\end{example}
\begin{figure}[t!]
\centering
\includegraphics[scale=0.7]{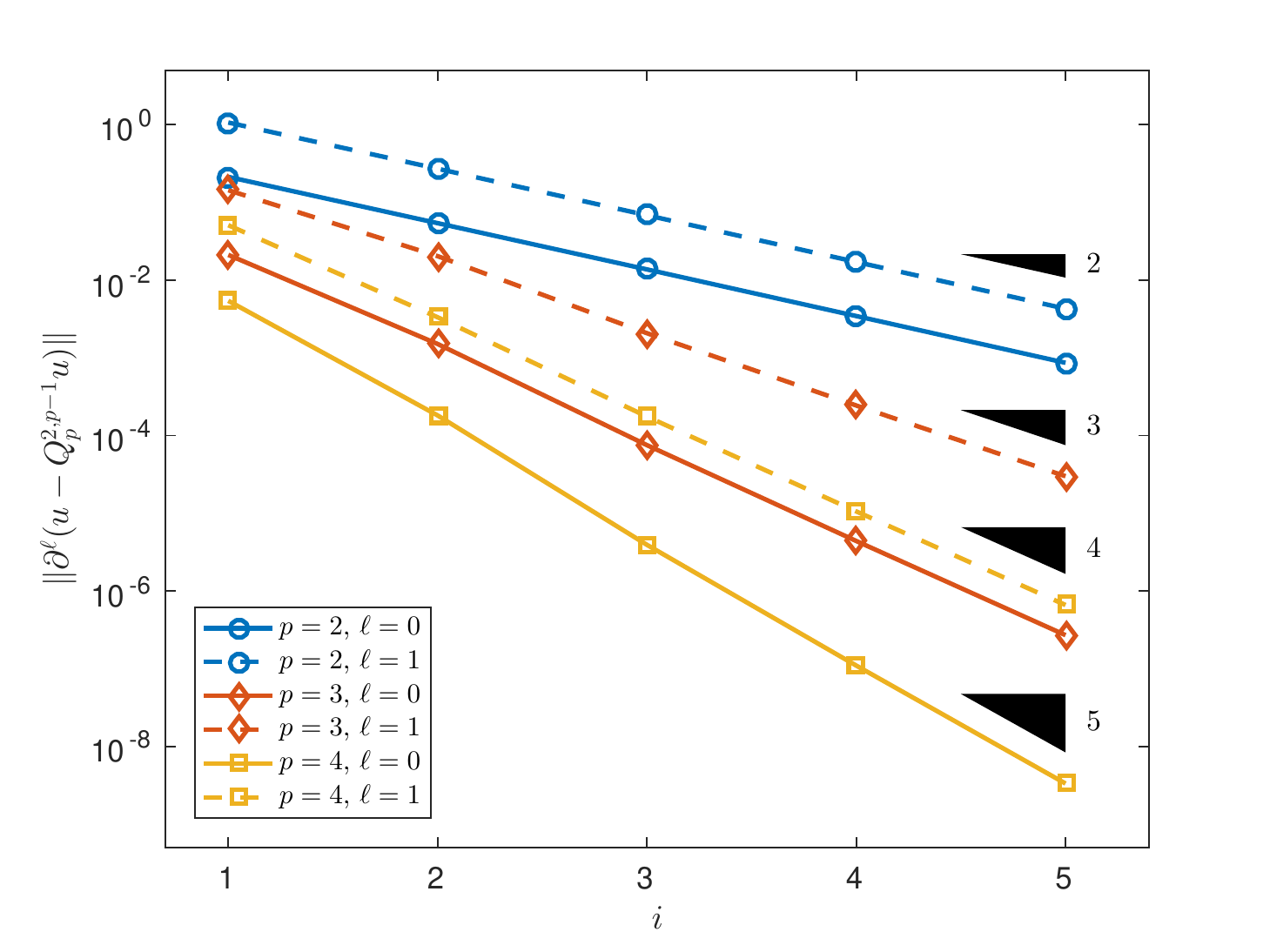}
\caption{The convergence of the error $\|\partial^{\ell}(u-Q^{2,p-1}_pu)\|$ for the problem specified in Example~\ref{ex:error}, taking a uniform knot sequence with $h=2^{-i}$, $i=1,\ldots,5$, and the different choices $p=2,3,4$ and $\ell=0,1$. The reference convergence order in $h$ is indicated by black triangles.}\label{fig:error}
\end{figure}

\begin{example}\label{ex:buffa}
If we choose $q=k+1$ in Theorem~\ref{thm:Qspline}, and consider all $\ell=0,\ldots,k+1$, then the condition on the degree becomes $p\geq \max\{r-1,2k+1\}$.
This is the same condition on the degree as in \cite[Theorem~2]{Buffa:11} and so their theorem is very similar to the special case $q=k+1$ of Theorem~\ref{thm:Qspline}. This is not surprising since the projector considered in \cite{Buffa:11} is defined by using the polynomial projector $Q^{k+1}_p:H^{k+1}(I_j)\to\mathcal{P}_p$ on each knot interval $I_j$, $j=0,\ldots,\nknots-1$, and then gluing them together with smoothness $k$ by means of the boundary interpolation of $Q^{k+1}_p$.
\end{example}

\begin{remark}\label{rem:outliers}
Similar to what was done in \cite[Section~3]{ManniSS:2022} for the Laplacian, the error estimate in \eqref{eq:spline-max-smooth} can be used to predict the number of outlier modes \cite{Cottrell:2006,Hiemstra:2021} in the numerical approximation of eigenvalues of higher order Laplacians with (full) Dirichlet boundary conditions, when using smooth spline spaces. As an example, consider the following eigenvalue problem: find $u_i\in H^2(0,1)$ and $\lambda_i\in\RR$, for $i=1,2,\ldots$, such that
\begin{align*}
u_i^{(4)}(x) &= \lambda_i u_i(x), \\
u_i(0)&=u_i(1)=u_i'(0)=u_i'(1)=0.
\end{align*} 
If we define the spline space $\mathcal{S}_{p,\knots,00}:=\{s\in\mathcal{S}^{p-1}_{p,\knots}:  s(0)=s(1)=s'(0)=s'(1)=0\}$ and let $n:=\dim \mathcal{S}_{p,\knots,00}$, then the above eigenvalue problem can be numerically approximated with the discrete weak formulation: find $u_{h,i}\in \mathcal{S}_{p,\knots,00}$ and $\lambda_{h,i}\in\RR$, for $i=1,\ldots,n$, such that
\begin{align}\label{eq:weak-eig}
(\partial^2 u_{h,i},\partial^2 v_h) = \lambda_{h,i} (u_{h,i},v_h), \quad \forall v_h\in \mathcal{S}_{p,\knots,00}.
\end{align}
The explicit error estimate in \eqref{eq:spline-max-smooth} can now be used to estimate the number of outlier modes, i.e., the number of eigenvalues which are badly approximated by \eqref{eq:weak-eig}. Using the same strategy as in \cite[Section~3]{ManniSS:2022}, we find that for fixed $h$, the discrete eigenvalue $\lambda_{h,i}$ is not an outlier if $h\lambda_i^{1/4}<\pi$. The main difference compared to the simpler case of the Laplacian in \cite[Section~3]{ManniSS:2022} is that $\lambda_i$ must now also be estimated; however, this can be done either analytically or numerically. Analytically, it is known that $\lambda_i$ is very close to $((2i+1)\pi/2)^4$ for $i=1,2,\ldots$; see, e.g., \cite[Section~5.5]{Davies:1995}.
Finally, we note that, using the techniques in \cite{ManniSS:2022}, this outlier discussion can also be extended to the higher dimensional tensor-product case.
\end{remark}

\subsection{Estimates for higher derivatives}

The $\ell$ in Theorem~\ref{thm:Qspline} only goes up to $q$, while the function $u$ is assumed to be in $H^r(a,b)$ for $r\geq q$. If $r>q$ then we can use an argument from \cite[Section~4]{Schultz:70} to obtain error estimates in the cases $q<\ell\leq r$. This argument provides an optimal rate of convergence in the maximum knot distance $h$, whenever the knot sequence is quasi-uniform, but convergence in the degree $p$ appears less than optimal. 
For any knot sequence $\knots$, let $\hmin$ denote its minimum knot distance. Since $\partial^{\ell}Q^{q,k}_pu$ is not in $L^2(a,b)$ for $\ell>k+1$ we define the broken norm $\|\cdot\|_{\knots}$ by
\begin{equation*}
\|\cdot\|_{\knots}^2 := \sum_{j=0}^{\nknots}\|\cdot\|_{L^2(I_j)}^2.
\end{equation*}

\begin{proposition}\label{pro:schultz}
Let $u\in H^r(a,b)$ be given.
For any degree $p$, knot sequence $\knots$ and smoothness $-1\leq k\leq p-1$, let $Q_p^{q,k}$ be the projector onto $\mathcal{S}^{k}_{p,\knots}$ defined in \eqref{eq:Qspline} for $q=0,\ldots,\min\{k+1,r-1\}$. Moreover, we set $m:=\max\{2r-q-1,p\}$. Then, for any $\ell=q+1,\ldots,r$, we have
\begin{align*}
&\|\partial^{\ell}(u-Q^{q,k}_pu)\|_{\knots} \\%\notag\\
&\quad \leq \left[c_{m-r,-1,r-\ell} + (c_{m-r,-1,r-q} + c_{p-q,k-q,r-q})\left(\frac{h}{h_{\min}}\right)^{\ell-q}\Biggl(\prod_{i=m-\ell+1}^{m-q} d_i\Biggr) \right]h^{r-\ell}\|\partial^ru\|,%\label{eq:schultz}
\end{align*}
for all $p\geq r-1$.
\end{proposition}
\begin{proof}
First, from Theorem~\ref{thm:Qspline} we deduce
\begin{align*}
\|\partial^{q}(Q^{r,r-1}_mu-Q^{q,k}_pu)\| &\leq \|\partial^{q}(Q^{r,r-1}_mu-u)\| + \|\partial^{q}(u-Q^{q,k}_pu)\|
\\
&\leq (c_{m-r,-1,r-q} + c_{p-q,k-q,r-q}) h^{r-q}\|\partial^ru\|,
\end{align*}
for $p\geq \max\{r-1,q-1\}=r-1$ and $m\geq\max\{r-1,2r-q-1\}=2r-q-1$.
Next, using the inverse inequality in Lemma~\ref{lem:inv-ineq-poly} on each knot interval, we obtain
\begin{equation*}
\|\partial^{\ell}(Q^{r,r-1}_mu-Q^{q,k}_pu)\|_{\knots} \leq \left(\frac{1}{h_{\min}}\right)^{\ell-q}\left(\prod_{i=m-\ell+1}^{m-q} d_i\right) \|\partial^{q}(Q^{r,r-1}_mu-Q^{q,k}_pu)\|.
\end{equation*}
The result now follows from
\begin{equation*}
\|\partial^{\ell}(u-Q^{q,k}_pu)\|_{\knots} \leq \|\partial^{\ell}(u-Q^{r,r-1}_mu)\| + \|\partial^{\ell}(Q^{r,r-1}_mu-Q^{q,k}_pu)\|_{\knots},
\end{equation*}
together with another application of Theorem~\ref{thm:Qspline}.
\end{proof}
We remark that the case $q=0$ and $\ell=1$ of the above proposition is a standard argument for showing error estimates (and stability) of the $L^2$-projection in the $H^1$-(semi)norm for quasi-uniform grids; see, e.g., the proof of \cite[Lemma~1]{Bank:1981} or \cite[Corollary~7.8]{Braess:2007}.

\subsection{Comparison with other projectors}\label{sec:Schultz}

%\begin{remark}\label{rmk:schultz}
Explicit error estimates for certain spline projectors with boundary interpolation have also been considered in \cite{Schultz:70}. We denote the sequence of projection operators onto $\mathcal{S}^{k}_{p,\knots}$ studied in \cite[Theorem~3.5]{Schultz:70} by $I_p^{q,k}$ for $q=1,\ldots,k+1$ and $p=2q-1$. Then, for any $\ell=0,\ldots,q$ and $r=2q$, it has been proved that
\begin{equation*}
\|\partial^{\ell}(u-I^{q,k}_pu)\| \leq K_{q,q,k,\ell} K_{q,q,k,0} h^{r-\ell}\|\partial^ru\|,
\end{equation*}
where
\begin{equation*}
  K_{q,q,k,\ell} := \begin{cases}
  1, &\ell=q,\\[0.25cm]
  \dfrac{(k+2-q)!}{(k+\ell+2-2q)!\,\pi^{q-\ell}}, & 2q-k-2 < \ell < q,\\[0.5cm]
  \dfrac{(k+2-q)!}{\pi^{q-\ell}}, &\ell\leq 2q-k-2.
\end{cases}
\end{equation*}
Numerical experiments reveal that the constants in Theorem~\ref{thm:Qspline} are never larger (and often much smaller) for the same values of $q,p,k,\ell,r$; see Example~\ref{ex:schultz} for a visual illustration.
Better constants for the projector $I_p^{q,k}$ can be found in \cite{Agarwal:94}, but they are not explicit in most cases. We note that the latter constants are explicit for maximal smoothness $k=p-1$ but these are still worse than $\left(\frac{1}{\pi}\right)^{r-\ell}$, the corresponding values attained in Theorem~\ref{thm:Qspline}. 
%\end{remark}
\begin{example}\label{ex:schultz}
In the case $\ell=0$, the constant in the error estimate of \cite{Schultz:70} is $(K_{q,q,k,0})^2$, while the one of Theorem~\ref{thm:Qspline} is $(c_{q-1,k-q,q})^2$ taking into account $r=2q$ and $p=2q-1$. To understand how these two expressions relate, we look at the quantity
\begin{equation}\label{ex:schultz-diff}
\log(K_{q,q,k,0})-\log(c_{q-1,k-q,q}).
\end{equation}
Several values are depicted in Figure~\ref{fig:schultz}. They are clearly nonnegative, implying that $K_{q,q,k,0} \geq c_{q-1,k-q,q}$.
\end{example}
\begin{figure}[t!]
\centering
\includegraphics[scale=0.7]{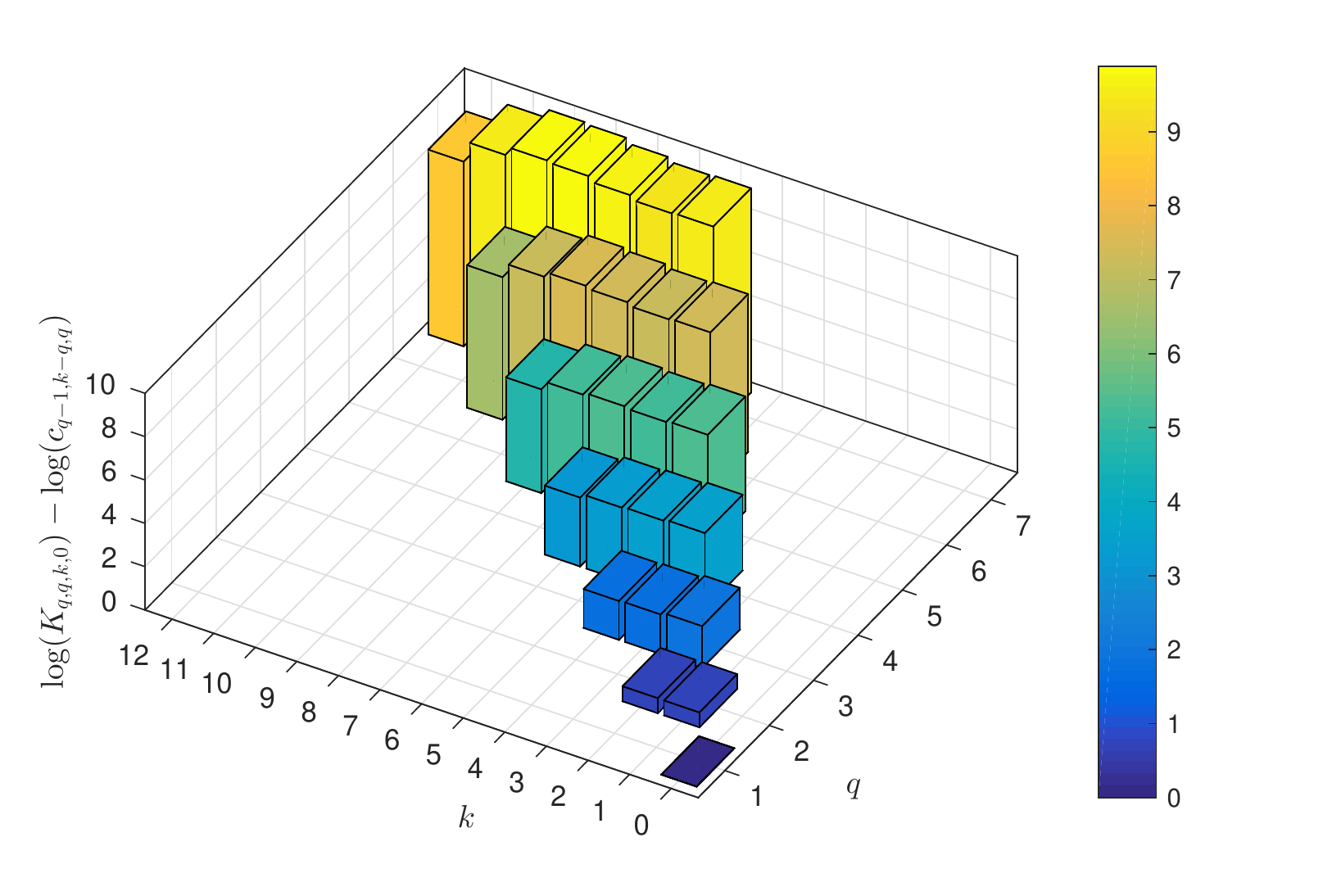}
\caption{The quantity in \eqref{ex:schultz-diff} for different choices of $q\geq1$ and $q-1\leq k\leq 2q-2$.}\label{fig:schultz}
\end{figure}

Let us now define the Ritz projector $R_p^{q,k}: H^q(a,b)\to \mathcal{S}^k_{p,\knots}$, for any $q=0,\ldots,k+1$, by
\begin{equation}\label{eq:Ritz-lower}
\begin{aligned}
(\partial^{q}R_p^{q,k} u,\partial^qv) &= (\partial^q u, \partial^q v), \quad &&\forall v\in \mathcal{S}^k_{p,\knots},
\\
(R_p^{q,k}u,g)&=(u,g), &&\forall g\in \mathcal{P}_{q-1}.
\end{aligned}
\end{equation}
This projector is a special case of \eqref{eq:Rproj}.
From Proposition \ref{pro:R-Q} we know that $R_p^{q,k}=Q_p^{q,k}$ whenever $p\geq 3q-1$. For degrees $2q-1\leq p < 3q-1$ the projectors $R_p^{q,k}$ and $Q_p^{q,k}$ are in general different, however, we can use Proposition~\ref{pro:R-Q-bound} (see \eqref{eq:R-error} and \eqref{eq:R-error-l}) to achieve an estimate for derivatives of the difference between them.

\begin{proposition}\label{pro:spline-supconv}
Let $u\in H^r(a,b)$ be given.
For any degree $p$, knot sequence $\knots$ and smoothness $-1\leq k\leq p-1$, let $Q_p^{q,k}$ and $R_p^{q,k}$ be the projectors onto $\mathcal{S}^{k}_{p,\knots}$ defined, respectively, in \eqref{eq:Qspline} and \eqref{eq:Ritz-lower} for $q=0,\ldots,\min\{k+1,r\}$. Then, we have 
\begin{equation*}
  \|R_p^{q,k}u-Q_p^{q,k}u\| \leq c_{p-q,k-q,q}c_{p-q,k-q,r-q}h^r \|\partial^ru\|,
\end{equation*}
for all $p\geq \max\{r-1,2q-1\}$.
Furthermore, for $\ell=1,\ldots,q-1$, we have
% \begin{equation}\label{eq:spline-supconv}
% \|\partial^\ell(R_p^{q,k}u-Q_p^{q,k}u)\|
% \leq \left(\frac{2\sqrt{3}}{b-a}\right)^{\ell}\left(\prod_{i=q-\ell}^{q-1} i^2\right)c_{p-q,k-q,q}c_{p-q,k-q,r-q}h^r \|\partial^ru\|,
% \end{equation}
\begin{equation}\label{eq:spline-supconv}
\|\partial^\ell(R_p^{q,k}u-Q_p^{q,k}u)\|
\leq c_{p-q,k-q,q}c_{p-q,k-q,r-q}h^r \left(\frac{1}{b-a}\right)^{\ell}\left(\prod_{i=q-\ell}^{q-1} d_i\right) \|\partial^ru\|,
\end{equation}
for all $p\geq \max\{r-1,2q-1\}$. 
Finally, for $\ell\geq q$, we have
\begin{equation*}
\|\partial^\ell(R_p^{q,k}u-Q_p^{q,k}u)\|=0.
\end{equation*}
\end{proposition}

We see that the estimate in \eqref{eq:spline-supconv} converges in $h$ with order $r$, independently of $\ell$. This is more than the general convergence order $r-\ell$ of the spline space for $\ell>0$; see Theorem~\ref{thm:Qspline}. Actually, Example~\ref{ex:error-diff} indicates that there is even a higher order of convergence for high $p$ than predicted by the proposition.

\begin{figure}[t!]
\centering
\includegraphics[scale=0.7]{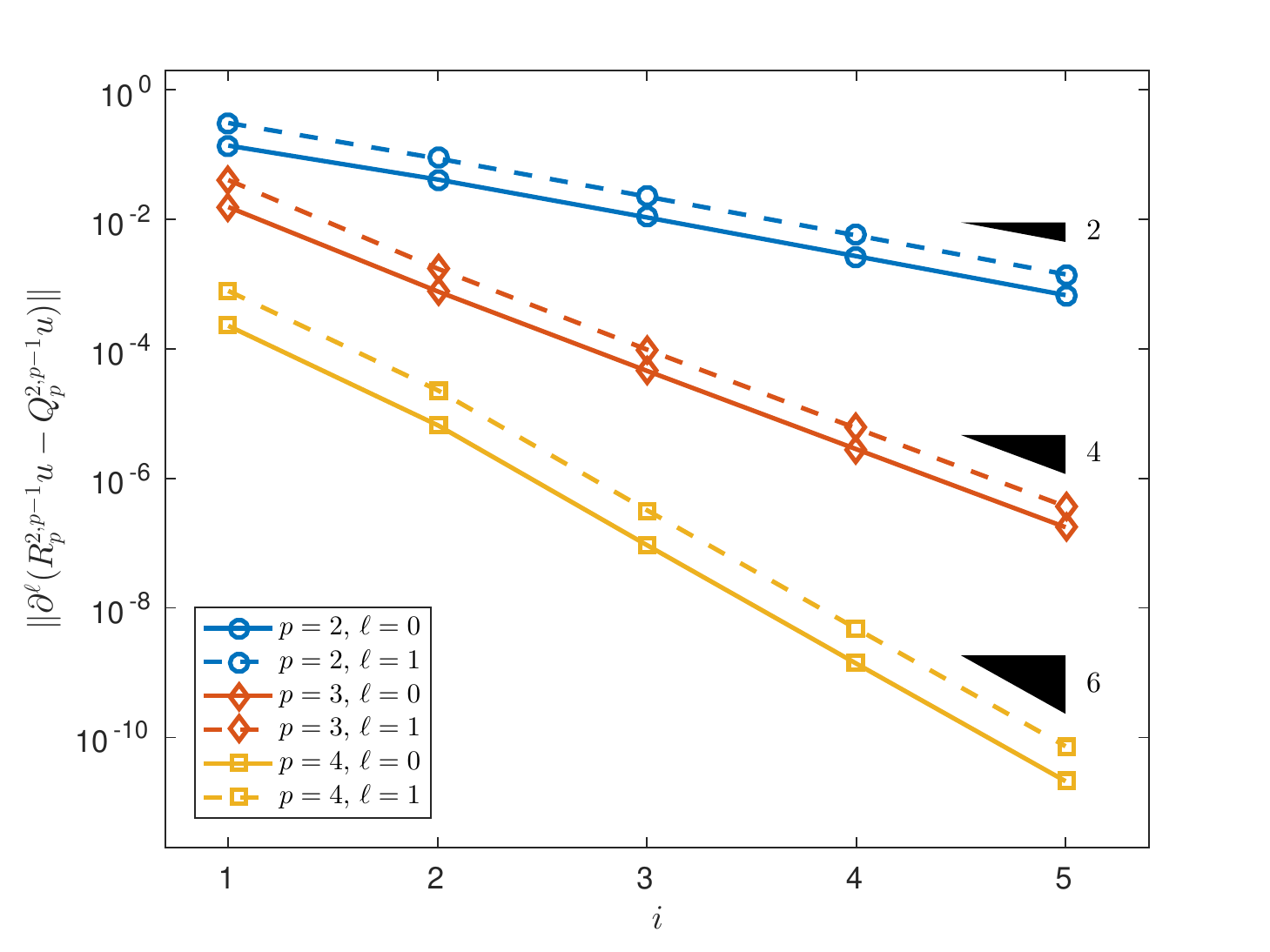}
\caption{The convergence of the error $\|\partial^{\ell}(R^{2,p-1}_pu-Q^{2,p-1}_pu)\|$ for the problem specified in Example~\ref{ex:error-diff}, taking a uniform knot sequence with $h=2^{-i}$, $i=1,\ldots,5$, and the different choices $p=2,3,4$ and $\ell=0,1$. The reference convergence order in $h$ is indicated by black triangles.}\label{fig:error-diff}
\end{figure}
\begin{example}\label{ex:error-diff}
Let $q=2$ and $[a,b]=[0,1]$. We choose again $u(x)=\sin(4x)$ as a continuation of Example~\ref{ex:error}.
Figure~\ref{fig:error-diff} shows the convergence behavior of $\|\partial^{\ell}(R^{2,p-1}_pu-Q^{2,p-1}_pu)\|$ for $p=2,3,4$ and $\ell=0,1$. As expected from Proposition~\ref{pro:spline-supconv}, we notice that the convergence order in $h$ does not depend on $\ell$ for a given $p$. However, the order seems to be equal to $2(p-q+1)$, which is better for $p>2q-1$ than the predicted order $r=p+1$.
\end{example}

%%%%%%%%%%%%%%%%%%%%
\section{Conclusions}\label{sec:conclusion}
%%%%%%%%%%%%%%%%%%%%
In an abstract framework, Ritz-type projectors with boundary interpolation properties have been presented and equipped with a priori error estimates. Their relation with the classical Ritz projectors has also been investigated.  

In the important case of projectors onto spline spaces, the provided error bounds are fully explicit in all the parameters defining the approximation space --- degree, smoothness, knot spacing --- and in the Sobolev regularity of the approximated function. They agree with those derived in \cite{Sande:2020} for the classical Ritz projectors and improve upon the error estimates known for typical spline projectors with Hermite interpolation properties at the boundary \cite{Schultz:70,Agarwal:94}.
In this perspective, the presented results enhance and complement those recently obtained for the classical Ritz projectors in \cite{Sande:2020} by enriching explicit spline error estimates with matching of the boundary data.

For the sake of brevity, the presentation has been confined to the onedimensional case. The multivariate extension of the proposed projectors (and of the corresponding error estimates) towards tensor-product structures is straightforward by following the same line of arguments already detailed in \cite{Sande:2020}.

The Hermite interpolation properties at the boundary make the presented projectors of interest in several applications as they directly allow for a local construction of globally smooth approximants by simply gluing local ones. Among the others, we mention their possible use in the context of isogeometric analysis for investigating the approximation properties in the frame of multi-degree spline spaces \cite{Speleers:2019,ToshniwalSHMH:2020} and smooth spline spaces on multi-patch geometries \cite{CollinST:2016,Sande:2020,ToshniwalSH:2017}. 

{Moreover, as mentioned in Remark \ref{rem:outliers}, for smooth spline spaces the Ritz-type projector with boundary interpolation and its corresponding error estimates can be used to predict the number of outlier modes in the numerical approximation of eigenvalues of higher order Laplacians with (full) Dirichlet boundary conditions.
}

Finally, we remark that both \eqref{eq:Qproj} and \eqref{eq:Rproj} are global projectors.  As future line of research, it might be interesting to quantify the difference between them and local projectors such as B\'ezier projection \cite{ThomasSETE:2015} and the classical ones in \cite{Lyche:18,LycheS1975}.

\section*{Acknowledgements}
This work was supported 
by the Beyond Borders Programme of the University of Rome Tor Vergata through the project ASTRID (CUP E84I19002250005)
and
by the MIUR Excellence Department Project awarded to the Department of Mathematics, University of Rome Tor Vergata (CUP E83C18000100006).
The authors are members of Gruppo Nazionale per il Calcolo Scientifico, Istituto Nazionale di Alta Matematica.

%%%%%%%%%%%%%%%%%%%%
\bibliography{nwidths}

\providecommand{\bysame}{\leavevmode\hbox to3em{\hrulefill}\thinspace}
\providecommand{\MR}{\relax\ifhmode\unskip\space\fi MR }
\providecommand{\MRhref}[2]{%
  \href{http://www.ams.org/mathscinet-getitem?mr=#1}{#2}
}
\providecommand{\href}[2]{#2}
\begin{thebibliography}{10}

\bibitem{Agarwal:94}
R.~P. Agarwal and P.~J.~Y. Wong, \emph{Explicit error bounds for the
  derivatives of spline interpolation in {$L_2$} norm}, Appl. Anal. \textbf{55}
  (1994), 189--205.

\bibitem{Ainsworth:2020}
M.~Ainsworth, O.~Davydov, and H.~Wang, \emph{Some remarks on spectral
  convergence and stability of iso-geometric analysis}, Comput. Methods Appl.
  Mech. Engrg. \textbf{372} (2020), 113408.

\bibitem{Bank:1981}
R.~E. Bank and T.~Dupont, \emph{An optimal order process for solving finite
  element equations}, Math. Comp. \textbf{36} (1981), 35--51.

\bibitem{Buffa:11}
L.~Beir\~ao~da Veiga, A.~Buffa, J.~Rivas, and G.~Sangalli, \emph{Some estimates
  for $h$-$p$-$k$-refinement in isogeometric analysis}, Numer. Math.
  \textbf{118} (2011), 271--305.

\bibitem{Braess:2007}
D.~Braess, \emph{Finite elements: Theory, fast solvers, and applications in
  solid mechanics}, third ed., Cambridge University Press, 2007.

\bibitem{Brenner:2008}
S.~C. Brenner and L.~R. Scott, \emph{The mathematical theory of finite element
  methods}, third ed., Texts in Applied Mathematics, vol.~15, Springer, 2008.

\bibitem{CollinST:2016}
A.~Collin, G.~Sangalli, and T.~Takacs, \emph{Analysis-suitable {$G^1$}
  multi-patch parametrizations for {$C^1$} isogeometric spaces}, Comput Aided
  Geom. Des. \textbf{47} (2016), 93--113.

\bibitem{Cottrell:2006}
J.~A. Cottrell, A.~Reali, Y.~Bazilevs, and T.~J.~R. Hughes, \emph{Isogeometric
  analysis of structural vibrations}, Comput. Methods Appl. Mech. Engrg.
  \textbf{195} (2006), 5257--5296.

\bibitem{Davies:1995}
E.~B. Davies, \emph{Spectral theory and differential operators}, Cambridge
  University Press, 1995.

\bibitem{Floater:2018}
M.~S. Floater and E.~Sande, \emph{Optimal spline spaces for {$L^2$} {$n$}-width
  problems with boundary conditions}, Constr. Approx. \textbf{50} (2019),
  1--18.

\bibitem{Goetgheluck:90}
P.~Goetgheluck, \emph{On the {M}arkov inequality in {$L^p$}-spaces}, J. Approx.
  Theory \textbf{62} (1990), 197--205.

\bibitem{Hiemstra:2021}
R.~R. Hiemstra, T.~J.~R. Hughes, A.~Reali, and D.~Schillinger, \emph{Removal of
  spurious outlier frequencies and modes from isogeometric discretizations of
  second- and fourth-order problems in one, two, and three dimensions}, Comput.
  Methods Appl. Mech. Engrg. \textbf{387} (2021), 114115.

\bibitem{Lyche:18}
T.~Lyche, C.~Manni, and H.~Speleers, \emph{Foundations of spline theory:
  {B}-splines, spline approximation, and hierarchical refinement}, Splines and
  {PDE}s: From Approximation Theory to Numerical Linear Algebra (T.~Lyche,
  C.~Manni, and H.~Speleers, eds.), Lecture Notes in Mathematics, vol. 2219,
  Springer International Publishing AG, 2018, pp.~1--76.

\bibitem{LycheS1975}
T.~Lyche and L.~L. Schumaker, \emph{Local spline approximation methods}, J.
  Approx. Theory \textbf{15} (1975), 294--325.

\bibitem{ManniSS:2022}
C.~Manni, E.~Sande, and H.~Speleers, \emph{Application of optimal spline
  subspaces for the removal of spurious outliers in isogeometric
  discretizations}, Comput. Methods Appl. Mech. Engrg. \textbf{389} (2022),
  114260.

\bibitem{Sande:2019}
E.~Sande, C.~Manni, and H.~Speleers, \emph{Sharp error estimates for spline
  approximation: Explicit constants, {$n$}-widths, and eigenfunction
  convergence}, Math. Models Methods Appl. Sci. \textbf{29} (2019), 1175--1205.

\bibitem{Sande:2020}
\bysame, \emph{Explicit error estimates for spline approximation of arbitrary
  smoothness in isogeometric analysis}, Numer. Math. \textbf{144} (2020),
  889--929.

\bibitem{Schultz:70}
M.~H. Schultz, \emph{Error bounds for polynomial spline interpolation}, Math.
  Comp. \textbf{24} (1970), 507--515.

\bibitem{Speleers:2019}
H.~Speleers, \emph{Algorithm 999: Computation of multi-degree {B}-splines}, ACM
  Trans. Math. Softw. \textbf{45} (2019), 43.

\bibitem{Takacs:2018}
S.~Takacs, \emph{Robust approximation error estimates and multigrid solvers for
  isogeometric multi-patch discretizations}, Math. Models Methods Appl. Sci.
  \textbf{28} (2018), 1899--1928.

\bibitem{Takacs:2016}
S.~Takacs and T.~Takacs, \emph{Approximation error estimates and inverse
  inequalities for {B}-splines of maximum smoothness}, Math. Models Methods
  Appl. Sci. \textbf{26} (2016), 1411--1445.

\bibitem{ThomasSETE:2015}
D.~C. Thomas, M.~A. Scott, J.~A. Evans, K.~Tew, and E.~J. Evans,
  \emph{{B}\'ezier projection: A unified approach for local projection and
  quadrature-free refinement and coarsening of {NURBS} and {T}-splines with
  particular application to isogeometric design and analysis}, Comput. Methods
  Appl. Mech. Engrg. \textbf{284} (2015), 55--105.

\bibitem{ToshniwalSHMH:2020}
D.~Toshniwal, H.~Speleers, R.~R. Hiemstra, C.~Manni, and T.~J.~R. Hughes,
  \emph{Multi-degree {B}-splines: Algorithmic computation and properties},
  Comput Aided Geom. Des. \textbf{76} (2020), 101792.

\bibitem{ToshniwalSH:2017}
D.~Toshniwal, H.~Speleers, and T.~J.~R. Hughes, \emph{Smooth cubic spline
  spaces on unstructured quadrilateral meshes with particular emphasis on
  extraordinary points: Geometric design and isogeometric analysis
  considerations}, Comput. Methods Appl. Mech. Engrg. \textbf{327} (2017),
  411--458.

\end{thebibliography}
%%%%%%%%%%%%%%%%%%%%

\end{document}